\documentclass[11pt, twoside, a4paper]{article}
\usepackage{amsmath, amsfonts, amssymb, amsthm}
\usepackage[latin1]{inputenc}
\usepackage[english]{babel}
\usepackage{graphicx}
\usepackage{psfrag}

\include{epsf}
\def\be{\begin{eqnarray*}}
\def\bel{\begin{eqnarray}}
\def\ee{\end{eqnarray*}}
\def\eel{\end{eqnarray}}

\def\N{\mathbb{N}}

\def\P{\mathbb{P}}

\def\C{\mathbb{C}}

\def\C{\mathcal{C}}

\def\Z{\mathbb{Z}}

\def\U{\mathcal{U}}

\def\1{(\begin{bf}0\end{bf},o)}
\def\2{(\begin{bf}0\end{bf},e)}

\def\vec{\left(\begin{array}{cc} }
\def\cev{\end{array} \right)}

\theoremstyle{plain}
\newtheorem{theo}{Theorem}[section]

\newtheorem{proposition}[theo]{Proposition}

\theoremstyle{definition}

\title{A note on the Poisson boundary of lamplighter random walks}
\author{Ecaterina Sava\footnote{Address: Institut f\"{u}r Mathematische Strukturtheorie, TU Graz,
Steyrergasse 30, 8010 Graz, Austria
\textit{e-mail address}: sava@tugraz.at}}

\begin{document}
\maketitle

\begin{abstract}
The main goal of this paper is to determine the Poisson boundary of lamplighter random walks over a general class of discrete groups $\Gamma$ endowed with a ``rich'' boundary. The starting point is the Strip Criterion of identification of the Poisson boundary for random walks on discrete groups due to Kaimanovich \cite{Kaimanovich2000}. A geometrical method for constructing the strip as a subset of the lamplighter group $\Z_{2}\wr\Gamma$ starting with a ``smaller'' strip in the group $\Gamma$ is developed. Then, this method is applied to several classes of base groups $\Gamma$: groups with infinitely many ends, hyperbolic groups in the sense of Gromov, and Euclidean lattices. We show that under suitable hypothesis the Poisson boundary for a class of random walks on lamplighter groups is the space of infinite limit configurations.
\paragraph{Keywords}
\end{abstract}


\section{Introduction}
Let $\Gamma$ be a finitely generated group, and imagine a lamp sitting at each group element. For simplicity, we consider that the lamps have only two states: $0$ (the lamp is swiched off) or $1$ (the lamp is swiched on), and initially all lamps are off. We think of a lamplighter person moving randomly in $\Gamma$ and switching randomly lamps on or off. We investigate the following model: at each step the lamplighter may walk to some random neighbour vertex, and may change the state of some lamps in a bounded neighbourhood of his position. This model can be interpreted as a random walk on the wreath product $(\Z/2\Z)\wr\Gamma$ governed by a probability measure $\mu$. The random walk is described by a transient Markov chain $Z_{n}$, which represents the random position of the lamplighter and the random configuration of the lamps at time $n$. We assume that the lamplighter random walk's projection on the base group $\Gamma$ is transient. Write $\Z_{2}:=\Z/2\Z$ and $G:=\Z_{2}\wr\Gamma$.

Transience of the projected random walk on $\Gamma$ implies that almost every path of the original random walk $Z_{n}$ on $G$ will leave behind a certain (infinitely supported) limit configuration on $\Gamma$. It is then natural to ask whether this limit configurations describe completely the behaviour of the random walk $Z_{n}$ at infinity.

For a more topological viewpoint, we attach to $G=\Z_{2}\wr\Gamma$ a natural boundary $\Omega$ at infinity, such that $G\cup\Omega$ is a metrizable space (not necessarily compact or complete) on which $G$ acts by homeomorphisms and every point in $\Omega$ is an accumulation point of a sequence in $G$. We then show that, in this topology, the random walk $Z_{n}$ converges almost surely to an $\Omega$-valued random variable, under the assumption that the projected random walk on $\Gamma$ converges to the boundary. If we denote by $\mu_{\infty}$ the limit distribution of $Z_{n}$ on $\Omega$, then the measure space $(\Omega,\mu_{\infty})$ provides a model for the behaviour at infinity of the random walk $Z_{n}$. We are interested if this space is maximal, i.e. there is no way (up to sets of measure $0$) of further refining this space. This maximal space is called the \textit{Poisson boundary} of the random walk. 

The Poisson boundary of a random walk on a group is a measure-theoretical space, which describes completely the significant behaviour of the random walk at infinity. Another way of defining the Poisson boundary is to say that it is the space of ergodic components of the time shift in the trajectory space. 

In order to prove that the measure space $(\Omega,\mu_{\infty})$ is indeed the Poisson boundary of the random walk $Z_{n}$, we shall use the very useful Strip Criterion of identification of the Poisson boundary due to Kaimanovich, which we state here in the most general form. For details see Kaimanovich \cite[Thm. $6.5$ on p. 677]{Kaimanovich2000} and \cite[Thm. $5.19$]{KaimanovichWoess2002}.
\begin{proposition}[\textbf{Strip Criterion}]\label{StripCriterion}
Let $\mu$ be a probability measure with finite first moment on $G$, and let $(B_{+},\lambda_{+})$ and $(B_{-},\lambda_{-})$ be $\mu$- and $\check{\mu}$-boundaries, respectively. If there exists a measurable $G$-equivariant map $S$ assigning to almost every pair of points $(b_{-},b_{+})\in B_{-} \times B_{+}$ a non-empty ``strip'' $ S(b_{-},b_{+})\subset G$, such that, for the ball $B(id,n)$ of radius $n$ in the metric of $G$,
\begin{equation*}
\frac{1}{n}\log| S(b_{-},b_{+})\cap B(id,n) | \to 0 ,\ \mbox{as}\ n\to\infty,
\end{equation*}
for $(\lambda_{-} \times \lambda_{+})$-almost every $(b_{-},b_{+})\in B_{-} \times B_{+}$, then $(B_{+},\lambda_{+})$ and $(B_{-},\lambda_{-})$ are the Poisson boundaries of the random walks $(G,\mu)$ and $(G,\check{\mu})$, respectively.
\end{proposition}

This criterion was applied by Kaimanovich to groups with sufficiently rich geometric boundaries, for which such strips have a natural geometric interpretation.

We shall give a general method for constructing the strip $S$ as a subset of the lamplighter group $\Z_{2}\wr \Gamma$, with the properties required in the Proposition \ref{StripCriterion}. This method requires that the Strip Criterion can be applied to the random walk on $\Gamma$. Also, some additional assumptions are required. The method can be applied to a large class of base groups $\Gamma$, which are endowed with a sufficiently rich boundary, so that the random walk on $\Gamma$ converges to this boundary. The important fact here is that the basic geometry for the lamplighter group $\Z_{2}\wr \Gamma$ is provided by the underlying structure $\Gamma$. We shall explain how this method works when $\Gamma$ is a group with infinitely many ends, a hyperbolic group, or a Euclidean lattice.

The paper is organized as follows. In Section \ref{sec:Lamplighter} we recall some definitions and basic properties of the main objects of study (lamplighter groups and random walks, wreath products). In Section \ref{sec:TheNaturalBoundary}, we attach both to the group $\Gamma$ and $\Z_{2}\wr\Gamma$ certain boundaries, which satisfy some required assumptions. Under the condition that the random walk on $\Gamma$ converges to the boundary almost surely, we prove that the random walk $Z_{n}$ on $\Z_{2}\wr\Gamma$ converges also to the boundary. In Section \ref{sec:The Poisson boundary}, we shall apply the Strip Criterion \ref{StripCriterion} in order to determine the Poisson boundary of random walks over a general class of groups $\Z_{2}\wr\Gamma$. We shall explain here the \textit{half-space method} for constructing a strip as a subset of $\Z_{2}\wr\Gamma$. The general procedure is based on the fact that the strip in the base group $\Gamma$ has additional ``nice'' properties, which help us to lift it to a bigger strip. We shall prove that the strip satisfies the required properties in Proposition \ref{StripCriterion}. Finally, we shall consider some typical examples of groups $\Gamma$, which are endowed with nice geometric boundaries, so that random walks on $\Gamma$ converge to this boundary. For this specific examples, we shall apply the \textit{half-space method}. 

Concluding the introduction, let us remark that the first to show that lamplighter groups are fascinating objects in the study of random walks were Kaimanovich and Vershik \cite{KaimanovichVershik1983}. By now, there is a considerable amount of literature on this topic. The paper of Kaimanovich \cite{KaimanovichVershik1983} may serve as a major source for the earlier literature. See also Lyons, Pemantle and Peres \cite{Lyons1996}, Erschler \cite{Erschler2001,Erschler2003}, Revelle \cite{Revelle,Revelle2003}, Pittet and Saloff-Coste \cite{Pittet&Saloff-Coste1996,Pittet&Saloff-Coste2002}, Grigorchuk and Zuk \cite{GrigorchukZuk2001}, Dicks and Schick \cite{DicksSchick2002}, Bartholdi and Woess \cite{BartholdiWoess2005}, Brofferio and Woess \cite{Brofferio 2006}.

\section{Lamplighter groups and random walks}\label{sec:Lamplighter}
\paragraph{Lamplighter groups.}Consider an infinite group $\Gamma$, generated by a finite set $S_{\Gamma}$. Denote by $e$ the identity element, and by $d(\cdot,\cdot)$ the word metric on $\Gamma$, that is, the length of the shortest path between two elements in the Cayley graph of $\Gamma$ (with respect to $S$).

Imagine a lamp sitting at each element of $\Gamma$, which can be switched off or on (encoded by 0 and 1). We think of a lamplighter person moving randomly in $\Gamma$ and switching randomly lamps on or off. At every moment of time the lamplighter will leave behind a certain configuration of lamps. The configurations of lamps are encoded by functions $\eta:\Gamma\rightarrow \Z_{2}$. We write $\hat{\mathcal{C}}=\{\eta : \Gamma\rightarrow\Z_{2}\}$ for the set of all configurations, and let $\mathcal{C} \subset \hat{\mathcal{C}}\ $ be the set of all finitely supported configurations, where a configuration is said to have finite support if the set $ supp(\eta)=\{x \in \Gamma : \eta(x)\ne  0  \} $ is finite. Denote by $\bf{0}\bf$ the zero configuration, i.e. the configuration which corresponds to all lamps switched off, and by $\delta_{x}$ the configuration where only the lamp at $x\in\Gamma$ is on and all other lamps are off.

Recall that the \textit{wreath product} of the groups $\Z_{2}$ and $\Gamma$ is a semidirect product of $\Gamma$ and the direct sum of copies of $\Z_{2}$ indexed by $\Gamma$, where every $x\in\Gamma$ acts on $\sum_{x\in\Gamma}\Z_{2}$ by the translation $T_{x}$ defined as 
\begin{equation*}
(T_{x}\eta)(y)=\eta(x^{-1}y),\forall y\in \Gamma.
\end{equation*}
Let $G:=\Z_{2}\wr\Gamma$ denote the \textit{wreath product}. The elements of $G$ are pairs of the form $(\eta,x)\in\C\times \Gamma$, where $\eta$ represents a (finitely supported!) configuration of the lamps and $x$ the position of the lamplighter. A group operation on $G$ is given by
\begin{equation*}
(\eta,x)(\eta^{'},x^{'})=(\eta\oplus T_{x}\eta^{'},xx^{'}),
\end{equation*}
where $x,x^{'}\in\Gamma$, $\eta,\eta^{'}\in\C$, $\oplus$ is the componentwise addition modulo $2$. The group identity is $(\textbf{0},e)$. We shall call $G$ together with this operation the \textit{lamplighter group} over $\Gamma$.
\paragraph{Lamplighter distance.}When $S_{\Gamma}$ is a generating set for $\Gamma$, then a natural set of generators for $G=\Z_{2}\wr\Gamma$ is given by 
\begin{equation*}
S_{G}=\{(\delta_{e},e),(\textbf{0},s):s\in S_{\Gamma}\}.
\end{equation*}
Consider the Cayley graph of $G$ with respect to the generating set $S_{G}$. We lift the word metric $d(\cdot,\cdot)$ on $\Gamma$ to a metric $d_{G}(\cdot,\cdot)$ on $G$ by assigning the following distances (lenghts) to the elements of $S_{G}$: $d_{G}((\textbf{0},e),(\textbf{0},s)):=1$ for $s\in S_{\Gamma}$ and $d_{G}((\textbf{0},e),(\delta_{e},e)):=c> 0$, where $c$ is some arbitrary, but fixed positive constant. Then the distance $d_{G}((\eta,x),(\eta^{'},x^{'}))$ between $(\eta,x)$ and $(\eta^{'},x^{'})$ is the length of the shortest path in the Cayley graph of $G$ joining these two vertices. More precisely, if we denote by $l(x,x^{'})$ the smallest length of a ``travelling salesman'' tour from $x$ to $x^{'}$ that visits each element of the set $\eta \bigtriangleup \eta ^{'}$ (where the two configurations are different), then   
\begin{equation}\label{GraphGmetric}
d_{G}((\eta,x),(\eta^{'},x^{'}))=l(x,x^{'})+c\cdot|\eta^{'}\bigtriangleup\eta|
\end{equation}
defines a metric on $G$.
\paragraph{Lamplighter random walks.}Let $\mu$ be a probability measure on $G$, such that $supp(\mu)$ generates $G$ as a group. Consider the random walk $Z_{n}$ on $G$ with one-step transition probabilites given by $p((\eta,x),(\eta^{'},x^{'}))=\mu((\eta,x)^{-1}(\eta^{'},x^{'}))$, starting at the identity $\2$. We shall call $Z_{n}$ the \textit{lamplighter random walk} over the base group $\Gamma$ and with law $\mu$. The lamplighter random walk starting at $\2$ can also be described by a sequence of $G$-valued random variables $Z_{n}$ in the following way:
\begin{equation}\label{lamplighter random variables}
 Z_{0}:=\2,\ Z_{n}=Z_{n-1}i_{n},\mbox{ for all } n\geq1,
\end{equation}
where $i_{n}$, with $i_{n}=(f_{n},z_{n})$, is a sequence of i.i.d. $G$-valued random variables governed by the probability measure $\mu$.
   
We write $Z_{n}=(\eta_{n},X_{n})$, where $\eta_{n}$ is the random configuration of lamps at time $n$ and $X_{n}$ is the random element of $\Gamma$ at which the lamplighter stands at time $n$. Therefore, the projection of $Z_{n}=(\eta_{n},X_{n})$ on $\Gamma$ is the random walk $X_{n}$ starting at the identity $e$ and with law 
\begin{equation*}
\nu(x)=\sum_{\eta\in\C}\mu(\eta,x).
\end{equation*}
The law $\nu$ of a random walk on $\Gamma$ is said to have \textit{finite first moment} if 
\begin{equation*}
\sum_{x\in \Gamma}d(e,x)\nu(x)<\infty.
\end{equation*} 
As a \textit{general assumption}, we assume the transience of $X_{n}$. By transience of $X_{n}$, every finite subset of $\Gamma$ is left with probability one after a finite time. Therefore, the sequence $(\eta_{n})_{n\in\N_{0}}$ of configurations converges pointwise to a random limit configuration $\eta_{\infty}$, which is not necessarily finitely supported. Now a natural question is whether the behaviour of the random walk at infinity is completely described by these limit configurations. For this purpose, a notion of "infinity" for the lamplighter group is needed.

\section{Convergence to the boundary}\label{sec:TheNaturalBoundary}

\paragraph{The boundary of $\Gamma$.}Consider the base group $\Gamma$ as above, with the word metric $d(\cdot,\cdot)$ on it. Let $\widehat{\Gamma}=\Gamma\cup\partial\Gamma$ be an extended space (not necessarily compact) with ideal \textit{boundary} $\partial\Gamma$ (the set of points at infinity), such that $\widehat{\Gamma}$ is compatible with the group structure on $\Gamma$ in the sense that the action of $\Gamma$ on itself extends to an action on $\widehat{\Gamma}$ by homeomorphisms. 

\paragraph{3.1} \label{sec:BasicAssumptions}\textbf{Basic assumptions.} Returning to the random walk $X_{n}$ on $\Gamma$, assume that:
\begin{description}
\item(a) The law $\nu$ of the random walk $X_{n}$ has finite first moment on $\Gamma$.
\item(b) The random walk $X_{n}$ converges almost surely to a random element of $\partial\Gamma$: there is a  	 $\partial\Gamma$-valued random variable $X_{\infty}$ such that in the topology of $\widehat{\Gamma}$,
\begin{equation*}
 \lim_{n\to\infty}X_{n}=X_{\infty},\mbox{ almost surely for every starting point } x\in\Gamma.
\end{equation*}
\item(c) The boundary $\partial\Gamma$ is such that the following \textbf{convergence property} holds: whenever $(x_{n}),(y_{n})$ are sequences in $\Gamma$ such that $(x_{n})$ accumulates at $\xi\in \partial\Gamma$ and 
\begin{equation}\label{Convergence}
\tag{\textbf{CP}}
d(x_{n},y_{n})/d(x_{n},e)\to 0\,\mbox{ as }n\to\infty, 
\end{equation} 
then $(y_{n})$ accumulates also at $\xi$.
\end{description}
\paragraph{The boundary of $G=\Z_{2}\wr\Gamma$.}Remark that the natural compactification of $\mathcal{C}$ in the topology of pointwise convergence is the set $\widehat{\mathcal{C}}$ of all, finitely or infinitely supported configurations. Since the vertex set of the Cayley graph of $G$ is $\mathcal{C}\times\Gamma$, the space $\partial G=(\widehat{\mathcal{C}}\times \widehat{\Gamma})\setminus (\mathcal{C} \times\Gamma)$ is a \textit{natural boundary} at infinity for $G$. Let us write $\widehat{G}=\widehat{\mathcal{C}} \times \widehat{\Gamma}$.

The boundary $\partial G$ contains many points towards the lamplighter random walk $Z_{n}$ does not converge, as we shall later see. For this reason, we define a ``smaller'' boundary $\Omega$ for the lamplighter group (which is still dense in $\partial G$), and we shall prove that the random walk $Z_{n}$ converges to a random variable with values in $\Omega$. Define
\begin{equation}\label{omega}
\Omega =\bigcup_{\mathfrak{u} \in \partial\Gamma}\mathcal{C}_{\mathfrak{u}}\times \{\mathfrak{u}\}, 
\end{equation}
where a configuration $\zeta$ is in $\mathcal{C}_{\mathfrak{u}}$ if and only if $\mathfrak{u}$ is its only accumulation point (i.e., there may be infinitely many lamps switched on only in a neighbourhood of $\mathfrak{u}$) or if $\zeta$ is finitely supported. The set $\mathcal{C}_{\mathfrak{u}}$ is dense in $\widehat{\mathcal{C}}$ because $\mathcal{C}\subset\mathcal{C}_{\mathfrak{u}}$ and $\mathcal{C}$ is dense in $\widehat{\mathcal{C}}$. Hence, $\Omega$ is also dense in $\partial G$.

The action of $G=\Z_{r}\wr\Gamma$ on itself extends to an action on $\widehat{G}$ by homeomorphisms and leaves the Borel subset $\Omega\subset\partial G$ invariant. If we take $(\eta,x)\in G$ and $(\zeta,\mathfrak{u}) \in \Omega$, then
\begin{equation*} \label{action}
(\eta,x)(\zeta,\mathfrak{u})=(\eta \oplus T_{x}\zeta,x \mathfrak{u}).
\end{equation*}
If $\mathfrak{u}\in\partial\Gamma$ and $\zeta$ is finitely supported or accumulates only at $\mathfrak{u}$, then $T_{x}\zeta$ can at most accumulate at $x\mathfrak{u}$. Also the configuration $\eta\oplus T_{x}\zeta$ accumulates again at most at $x\mathfrak{u}$ because $\eta$ is finitely supported, so that adding $\eta$ modifies $T_{x}\zeta$ only in finitely many points.

When the \textbf{Basic assumptions $(3.1)$} hold, we are able to state the first result on convergence of the lamplighter random walk $Z_{n}$ to $\Omega$.
\begin{theo}\label{ConvergenceTheorem}
Let $Z_{n}=(\eta_{n},X_{n})$ be a random walk with law $\mu$ on the group $G=\Z_{r}\wr\Gamma$ such that $supp(\mu)$ generates $G$. If $\Omega$ is defined as in \eqref{omega} and $\mu$ has finite first moment, then there exists an $\Omega$-valued random variable $Z_{\infty}=(\eta_{\infty},X_{\infty})$ such that $Z_{n}\to Z_{\infty}$ almost surely, for every starting point. Moreover the distribution of $Z_{\infty}$ is a continuous measure on $\Omega$.
\end{theo}
\begin{proof}
Without loss of generality, we may suppose that the starting point for the lamplighter random walk $Z_{n}$ is $id=\2$, that is, we start a random walk in the identity element $e$ of $\Gamma$ with all the lamps switched off.

The support of $\mu$ generates $G=\Z_{r}\wr\Gamma$ as a group, therefore, also the law $\nu$ of the projected random walk $X_{n}$ on $\Gamma$ is such that its support generates $\Gamma$. By assumption, the random walk $X_{n}$ on $\Gamma$ is transient, and converges almost surely to a random variable $X_{\infty}\in\partial\Gamma$. 

Now, assume that the lamplighter random walk $Z_{n}$ has finite first moment. Let $(y_{n})$ be an unbounded sequence of elements in $\Gamma$, such that $y_{n}\in supp(T_{X_{n-1}}f_{n})$ ($f_{n}$ is a finitely supported configuration), for each $n$, that is, $y_{n}$ is a group element where the lamp is switched on. Since, by assumptions, the law $\nu$ of the random walk $X_{n}$ on $\Gamma$ has finite first moment, the following holds with probability $1$:
\begin{equation*}
d(y_{n},X_{n-1})/{n}\to 0,\mbox{ as }n\to\infty.
\end{equation*}
Moreover, it follows from \textit{Kingman's subadditive ergodic theorem} (see Kingman\cite{Kingman1968}) that there exists finite constant  $m>0$ such that
\begin{equation*}
d(X_{n},e)/n\to m,\ \mbox{as}\ n\to\infty,\ \mbox{almost surely}.
\end{equation*}
Using the last two equations and the triangle inequality, we have $d(X_{n},y_{n})/d(X_{n},e)\to 0$, as $n\to\infty$. Recall now that the boundary $\partial\Gamma$ satisfies the convergence property \eqref{Convergence} and $X_{n}\to X_{\infty}\in\partial\Gamma$. Therefore, the sequence $(y_{n})$ accumulates at $X_{\infty}$. Now, from the definition of the group operation on $G$ and the equation \eqref{lamplighter random variables}, one can remark that the configuration $\eta_{i}$ of lamps at every moment of time $i$ is obtained by adding (componentwise addition modulo 2) to the configuration $\eta_{i-1}$ the configuration $T_{X_{i-1}}f_{i}$  (where $f_{i}$  is finitely supported). Hence
\begin{equation*}
supp(\eta_{n})\subset \bigcup_{i=1}^{n}supp(T_{X_{i-1}}f_{i}),
\end{equation*}
which is a union of finite sets. From the above, the sequence $y_{n}\in supp(T_{X_{n-1}}f_{n})$ accumulates at $X_{\infty}$, therefore the sequence $supp(\eta_{n})$ must accumulate at $X_{\infty}$. That is, the random configuration $\eta_{n}$ converges pointwise to a configuration $\eta_{\infty}$, which accumulates at $X_{\infty}$ and $Z_{n}=(\eta_{n},X_{n})$ converges to a random element $Z_{\infty}=(\eta_{\infty},X_{\infty})\in\Omega$.

When the limit distribution of $X_{n}$ is a continuous measure on $\partial\Gamma$ (i.e., it carries no point mass), then the same is true for the limit distribution of $Z_{n}=(\eta_{n},X_{n})$ on $\Omega$. Otherwise, supposing that there exists some single point of $\Omega$ with non-zero measure, then one comes to a contradiction finding some single point in $\partial\Gamma$ with non-zero measure, which is not true because of the continuity of the limit distribution on $\partial\Gamma$.

Even when the limit distribution $\nu_{\infty}$ of $X_{n}$ is not a continuous measure on $\partial\Gamma$, the limit distribution of $Z_{n}$ is still continuous. When the measure $\nu_{\infty}$ is not continuous, there exists $\mathfrak{u}\in\partial\Gamma$, with $\nu_{\infty}(\{\mathfrak{u}\})=\mathbb{P}[X_{n}=\mathfrak{u}|X_{o}=e]>0$. Assume that the limit distribution $\mu_{\infty}$ of $Z_{n}$ is not continuous. Then, there is a configuration $\phi$ such that, the limit configuration of the lamplighter random walk $Z_{n}$ is $\phi$. Then, for every $x\in\Gamma$, all trajectories of the random walk $X_{n}$ starting at $x$ and converging to the deterministic boundary element $\mathfrak{u}$ will have the same limit configuration $\phi_{x}$, accumulating only at $\mathfrak{u}$. Note that the group $\Gamma$ acts also on the space of limit configurations by translations, and for every $y\in supp(\nu)$, $T_{y}\phi_{x}=\phi_{xy}$. Since the support of $\nu$ generates $\Gamma$, this can not happen. Therefore, the distribution of $Z_{n}$ is a continuous measure. One can also prove the continuity of the limit distribution using Borel-Cantelli lemma.
\end{proof}

\section{The half-space method and the Poisson boundary of the lamplighter random walk}\label{sec:The Poisson boundary}

Under the assumptions of Theorem \ref{ConvergenceTheorem}, let $\mu_{\infty}$ be the distribution of $Z_{\infty}$ on $\Omega$, given that the position of the random walk $Z_{n}$ at time $n=0$ is $id=\2$. This is a probability measure on $\Omega$ defined for Borel sets $U\subset\Omega$ by
\begin{equation*}
\mu_{\infty}(U)=\P[Z_{\infty}\in U\vert Z_{0}=\2].
\end{equation*}
The measure $\mu_{\infty}$ is a \textit{harmonic measure} for the random walk $Z_{n}$ with law $\mu$, that is, it satisfies the convolution equation $\mu \ast\mu_{\infty}=\mu_{\infty}$. Since $G$ acts on $\Omega$ by measurable bijections and the measure $\mu_{\infty}$ is stationary with respect to $\mu$, it follows that $(\Omega,\mu_{\infty})$ is a \textit{$\mu$-boundary} (or a \textit{Furstenberg boundary}) for the random walk $Z_{n}$ with law $\mu$, in the sense of Furstenberg \cite{Furstenberg}. There exists a maximal $\mu$-boundary, which is called the \textit{Poisson boundary} for the random walk $Z_{n}$.

The typical situation when a $\mu$-boundary $(B,\lambda)$ can arise is when $B$ is a certain topological or combinatorial boundary of a group $G$, and almost all paths of the random walk $Z_{n}$ with law $\mu$ on $G$ converge (in a certain sense which needs to be specified in each particular case) to a limit point $Z_{\infty}\in B$. Then the space $B$ considered as a measure space with the resulting hitting distribution $\lambda$ (the \textit{harmonic measure} on $B$) is a $\mu$-boundary of the random walk $Z_{n}$.

We want to know if the measure space $(\Omega,\mu_{\infty})$ is indeed maximal. In order to check the maximality of the $\mu$-boundary, we use the very useful Strip Criterion \ref{StripCriterion} of identification of the Poisson boundary due to Kaimanovich \cite[Thm. $6.5$ p. 677]{Kaimanovich2000} and \cite[Thm. $5.19$]{KaimanovichWoess2002}. This criterion is symmetric with respect to the time reversal and leads to a simultaneous identification of the Poisson boundary of the random walk and of the reflected random walk, respectively. Consider now the \textit{reflected random walk} $\check{Z}_{n}=(\check{\eta}_{n},\check{X}_{n})$ on $G$ with law $\check{\mu}(g)=\mu(g^{-1})$ for all $g\in G$, and starting at $\2$. The \textit{reflected random walk} $\check{X}_{n}$ on $\Gamma$ is the random walk on $\Gamma$ with law $\check{\nu}(x)=\nu(x^{-1})$, for all $x\in\Gamma$, and starting at $e$.

\paragraph*{\begin{center}The half-space method\end{center}}\label{sec:Method for constructing the strip}
Assume that:
\begin{enumerate}
\item The \textbf{Basic assumptions $(3.1)$} hold for $X_{n}$ and $\check{X}_{n}$. Let $\nu_{\infty}$ and $\check{\nu}_{\infty}$ be the respective hitting distributions on $\Gamma$. 
\item For $\nu_{\infty}\times\check{\nu}_{\infty}$-almost every pair $(\mathfrak{u},\mathfrak{v})\in\partial \Gamma\times\partial\Gamma$, one has a strip $\mathfrak{s}(\mathfrak{u},\mathfrak{v})$ which satisfies the conditions from the Proposition \ref{StripCriterion}. That is, it is a subset of $\Gamma$, it is $\Gamma$-equivariant, and it has subexponential growth, i.e.,
\begin{equation}\label{SubexponentialGrowthBaseStrip}
\dfrac{1}{n}\log|\mathfrak{s}(\mathfrak{u},\mathfrak{v})\cap B(e,n)|\to 0,\mbox{ as }n\to\infty,
\end{equation}
where $B(e,n)=\{x\in\Gamma:\ d(e,x)\leq n\}$ is the ball with center $o$ and radius $n$ in $\Gamma$. 
\item For every $x\in\mathfrak{s}(\mathfrak{u},\mathfrak{v})$, one can assign to the triple $(\mathfrak{u},\mathfrak{v},x)$ a partiton of $\Gamma$ into \textit{half-spaces} $\Gamma_{\pm}$ such that $\Gamma_{+}$ (respectively, $\Gamma_{-}$) contains a neighbourhood of $\mathfrak{u}$ (respectively, $\mathfrak{v}$), and the assignments $(\mathfrak{u},\mathfrak{v},x)\mapsto\Gamma_{\pm}$ are $\Gamma$-equivariant.
\end{enumerate}
In the last item, one can partition $\Gamma$ in more that two subsets and the method can still be applied. However, the important subsets are that ones containing a neighbourhood of $\mathfrak{u}$ (respectively, $\mathfrak{v}$), since only there may be infinitely many lamps switched on (because $\mathfrak{u}$ and $\mathfrak{v}$ are the respective boundary points toward the random walks $X_{n}$ and $\check{X}_{n}$ converge). What we want is to build a finitely supported configuration associated to pairs $(\phi_{+},\phi_{-})$ of limit configurations (of the lamplighter random walk and of the reflected random walk) accumulating at $\mathfrak{u}$ and $\mathfrak{v}$, respectively. In order to do this we restrict $\phi_{+}$ and $\phi_{-}$ on $\Gamma_{-}$ and $\Gamma_{+}$, respectively, and then we ``glue together'' the restrictions. Since the new configuration depends on the partition of $\Gamma$, we cannot choose the same partition for  all x, because we will have a constant configuration which is not equivariant. Therefore, the partition of $\Gamma$ should depend on $x$. 

We state here the main result of this paper.
\begin{theo}\label{PoissonTheorem}
Let $Z_{n}=(\eta_{n},X_{n})$ be a random walk with law $\mu$ on $G=\Z_{2}\wr\Gamma$, such that $supp(\mu)$ generates $G$. Suppose that $\mu$ has finite first moment and $\Omega$ is defined as in \eqref{omega}. If the above assumptions are satisfied, then the measure space $(\Omega,\mu_{\infty})$ is the Poisson boundary of $Z_{n}$, where $\mu_{\infty}$ is the limit distribution on $\Omega$ of $Z_{n}$ starting at $id=\2$.
\end{theo}
\begin{proof}
In order to apply the Strip Criterion \ref{StripCriterion}, we need to find $\mu$- and $\check{\mu}$-boundaries for the lamplighter random walk $Z_{n}$ and the reflected lamplighter random walk $\check{Z}_{n}$, respectively. By Theorem \ref{ConvergenceTheorem} each of the random walks $Z_{n}$ and $\check{Z}_{n}$ starting at $id$ converges almost surely to an $\Omega$-valued random variable. If $\mu_{\infty}$  and $\check{\mu}_{\infty}$ are their respective limit distributions on $\Omega$, then the spaces $(\Omega,\mu_{\infty})$ and $(\Omega,\check{\mu}_{\infty})$ are $\mu$- and $\check{\mu}$- boundaries of the respective random walks. 
 
Let us take $b_{+}=(\phi_{+},\mathfrak{u})$, $b_{-}=(\phi_{-},\mathfrak{v})\in\Omega$, where $\phi_{+}$ and $\phi_{-}$ are the limit configurations of $Z_{n}$ and $\check{Z}_{n}$, respectively, and $\mathfrak{u},\mathfrak{v}\in\partial\Gamma$ are their only respective accumulation points. By the continuity of $\mu_{\infty}$  and $\check{\mu}_{\infty}$, the set $\{(b_{+},b_{-})\in\Omega\times\Omega:\mathfrak{u}=\mathfrak{v}\}$ has $(\mu_{\infty}\times\check{\mu}_{\infty})$-measure $0$, so that, in constructing the strip $S(b_{+},b_{-})$ we shall consider only the case $\mathfrak{u}\neq\mathfrak{v}$.

Using the third item in the above assumptions, let us consider a partition of $\Gamma$ into $\Gamma_{+}$, $\Gamma_{-}$, and eventually $\Gamma\setminus(\Gamma_{+}\cup\Gamma_{-})$, where $\Gamma_{+}$ (respectively, $\Gamma_{-}$) contains a neighbourhood of $\mathfrak{u}$ (respectively, $\mathfrak{v}$), and $\Gamma\setminus(\Gamma_{+}\cup\Gamma_{-})$ is the remaining subset (which may be empty). The set $\Gamma\setminus(\Gamma_{+}\cup\Gamma_{-})$ contains neither $\mathfrak{u}$ nor $\mathfrak{v}$. The restriction of $\phi_{+}$ on $\Gamma_{-}$ (respectively, of $\phi_{-}$ on $\Gamma_{+}$) is finitely supported since its only accumulation point is $\mathfrak{u}$ (respectively, $\mathfrak{v}$), which is not in a neighbourhood of $\Gamma_{-}$ (respectively, $\Gamma_{+}$). Now ``put together'' the restriction of $\phi_{+}$ on $\Gamma_{-}$ and of $\phi_{-}$ on $\Gamma_{-}$ in order to get the new configuration
\begin{equation}\label{StripConfiguration}
\Phi(b_{+},b_{-},x)=
\begin{cases}
\phi_{-}, & \mbox{on}\  \Gamma_{+}\\
\phi_{+}, & \mbox{on}\ \Gamma_{-}\\
0, & \mbox{on}\ \Gamma\setminus(\Gamma_{+}\cup\Gamma_{-})
\end{cases} 
\end{equation}
on $\Gamma$, which is, by construction, finitely supported. Now, the sought for the ``bigger'' strip $S(b_{+},b_{-})\subset G$ is the set
\begin{equation}\label{LamplighterStrip}
S(b_{+},b_{-})=\{\left(\Phi,x\right) :\  x\in\mathfrak{s}(\mathfrak{u},\mathfrak{v})\}
\end{equation}
of all pairs $(\Phi,x)$, where $\Phi=\Phi(b_{+},b_{-},x)$ is the configuration defined above and $x$ runs through the strip $\mathfrak{s}(\mathfrak{u},\mathfrak{v})$ in $\Gamma$. This is a subset of $G=\Z_{2}\wr\Gamma$. We prove that the map $(b_{+},b_{-})\mapsto S(b_{+},b_{-})$ is $G$-equivariant, i.e., for $g=(\eta,\gamma)\in G$:
\begin{equation*}
gS(b_{+},b_{-})=S(gb_{+},gb_{-}).
\end{equation*}
Next,
\begin{equation*}
gS(b_{+},b_{-})=(\eta,\gamma)\cdot\{\left(\Phi,x\right) :\  x\in\mathfrak{s}(\mathfrak{u},\mathfrak{v})\}=\left\lbrace (\eta\oplus T_{\gamma}\Phi,\gamma x),\ x\in\mathfrak{s}(\mathfrak{u},\mathfrak{v})\right\rbrace .
\end{equation*}
If $x\in\mathfrak{s}(\mathfrak{u},\mathfrak{v})$, then $\gamma x\in\mathfrak{s}(\gamma \mathfrak{u},\gamma\mathfrak{v})$, since $\mathfrak{s}(\gamma \mathfrak{u},\gamma\mathfrak{v})$ is $\Gamma$-equivariant. Also,
\begin{equation*}
\eta\oplus T_{\gamma}\Phi=
\begin{cases}
\eta\oplus T_{\gamma}\phi_{-}, & \mbox{on}\  \gamma\Gamma_{+} \\
\eta\oplus T_{\gamma}\phi_{+}, & \mbox{on}\  \gamma\Gamma_{-} \\
	0, & \mbox{on}\ \Gamma\setminus(\gamma\Gamma_{+}\cup \gamma\Gamma_{-}).
\end{cases}
\end{equation*}
This means that $\eta\oplus T_{\gamma}\Phi(b_{+},b_{-},x)=\Phi(gb_{+},gb_{-},\gamma x),\forall x\in\mathfrak{s}(\mathfrak{u},\mathfrak{v})$. On the other side,
\begin{equation*}
S(gb_{+},gb_{-})=S((\eta\oplus T_{\gamma}\phi_{+},\gamma\mathfrak{u}),(\eta\oplus T_{\gamma}\phi_{-},\gamma\mathfrak{v}))=\{(\Phi(gb_{+},gb_{-},\gamma x),\gamma x),\ x\in\mathfrak{s}(\mathfrak{u},\mathfrak{v})\},
\end{equation*}
that is, $gS(b_{+},b_{-})=S(gb_{+},gb_{-})$, and this  proves the $G$-equivariance of the strip $S(b_{+},b_{-})$.

Finally, let us prove that the strip $S(b_{+},b_{-})$ has subexponential growth. For this, let $(\eta,x)\in S(b_{+},b_{-})$ such that $d_{G}(\2,(\eta,x))\leq n$. From the definition of the metrics $d_{G}(\cdot,\cdot)$ and $d(\cdot,\cdot)$ on $G$ and $\Gamma$, respectively, it follows that $d(e,x)\leq n$. Therefore, if 
\begin{equation}\label{small_strip}
(\eta,x)\in S(b_{+},b_{-})\cap B(\2,n),\mbox{ then }x\in\mathfrak{s}(\mathfrak{u},\mathfrak{v})\cap B(e,n),
\end{equation}
where $B(\2,n)$ (respectively, $B(e,n)$) is the ball with center $id=\2$ (respectively, $e$) and of radius $n$ in $G$ (respectively, $\Gamma$). Since for every $x\in\mathfrak{s}(\mathfrak{u},\mathfrak{v})$ we associate only one configuration $\Phi$ in $S(b_{+},b_{-})$, equation \eqref{small_strip} implies that 
\begin{equation*}
|S(b_{+},b_{-})\cap B(\2,n)|\leq |\mathfrak{s}(\mathfrak{u},\mathfrak{v})\cap B(e,n)|.
\end{equation*}
Now, the assumption \eqref{SubexponentialGrowthBaseStrip} that the $\mathfrak{s}(\mathfrak{u},\mathfrak{v})$ has subexponential growth leads to 
\begin{equation*}
\dfrac{\log|S(b_{+},b_{-})\cap B(\2,n)|}{n}\to 0,\mbox{ as }n\to\infty, 
\end{equation*}
and this proves the subexponential growth of the strip.
Since for almost every pair of points $(b_{+},b_{-})\in\Omega\times\Omega$ we have assigned a strip $S(b_{+},b_{-})$ which satisfies the conditions from Proposition \ref{StripCriterion}, it follows that the measure space $(\Omega,\mu_{\infty})$ is the Poisson boundary of the lamplighter random walk $Z_{n}$.
\end{proof}
As an application of the half-space method, we consider several classes of base groups $\Gamma$: groups with infinitely many ends, hyperbolic groups and Euclidean lattices.

\subsection{Groups with infinitely many ends}\label{GraphsWithInfEnds}

The concept of \textit{ends} in the discrete settings goes back to Freudenthal \cite{Freudenthal1944}. The \textit{space of ends} $\partial\Gamma$ of a finitely generated group $\Gamma$ is defined as the space of ends of its Cayley graph with respect to a certain finite generating set. Recall that an end in a graph is an equivalence class of one-sided infinite paths, where two such paths are equivalent if there is a third one which meets each of the two infinitely often. We omit the description of the topology of $\widehat{\Gamma}=\Gamma\cup\partial\Gamma$, which can be found in Woess \cite{woess}. The space $\widehat{\Gamma}=\Gamma\cup\partial\Gamma$ is called the \textit{end compactification} of $\Gamma$.

A finitely generated group $\Gamma$, has one, two or infinitely many ends. If it has one end, then the end compactification is not suitable for a good description of the structure of $\Gamma$ at infinity. If it has two ends, then it is quasi-isometric with the two-way-infinite path and the Poisson boundary of any random walk with finite first moment on $\Gamma$ is trivial (see Woess \cite[Thm $25.4$]{woess}). Thus, we shall consider here the case when the underlying group $\Gamma$ for the lamplighter random walk $Z_{n}$ is a \textit{group with infinitely many ends}. The natural geometric boundary of $\Gamma$ is the space of ends $\partial\Gamma$.

We shall also use the powerful theory of cuts and structure trees developed by Dunwoody; see the book by Dicks and Dunwoody \cite{Dunwoody&Dicks}, or for another detailed description see Woess \cite{woess} and Thomassen and Woess \cite{Thomassen&Woess}. A detailed study of structure theory may be very fruitful for obtaining information on the behaviour of the random walks. We shall again omit the description of the structure tree $\mathcal{T}$ of the Cayley graph of a finitely generated group $\Gamma$ and of the structure map $\varphi$ between the Cayley graph of $\Gamma$ and its structure tree. The structure tree $\mathcal{T}$ is countable, but not necessarily locally finite, and $\Gamma$ acts on $\mathcal{T}$  by automorphisms.

The following result is a particular case of Theorem \ref{PoissonTheorem}. This is the first example where the half-space method can be applied in order to find the Poisson boundary of lamplighter random walks over groups with infinitely many ends.
\begin{theo}\label{PoissonInfEnds}
Let $\Gamma$ be a group with infinitely many ends and $Z_{n}=(\eta_{n},X_{n})$ be a random walk with law $\mu$ on $G=\Z_{2}\wr \Gamma$, such that $supp(\mu)$ generates $G$. Suppose that $\mu$ has finite first moment and $\Omega$ is defined as in \eqref{omega}. If $\mu_{\infty}$ is the limit distribution on $\Omega$ of the random walk $Z_{n}$ starting at $id=\2$, then $(\Omega,\mu_{\infty})$ is the Poisson boundary of the random walk $Z_{n}$. 
\end{theo}
\begin{proof}
In order to apply the Theorem \ref{PoissonTheorem}, we show that the conditions required in the half-space method are satisfied for the base group $\Gamma$ and random walks $X_{n}$ on $\Gamma$. First of all, one can check that the space of ends $\partial\Gamma$ satisfies the convergence property (\ref{Convergence}). When $\Gamma$ is a group with infinitely many ends and the law $\nu$ of the random walk $X_{n}$ has finite first moment, by Woess \cite{WoessAmenable1989}, $X_{n}$ converges in the end topology to a random end from $\partial\Gamma$. The same is true for the random walk $\check{X}_{n}$ with law $\check{\nu}$ on $\Gamma$. Moreover, the limit distributions are continuous measures on $\partial\Gamma$. Let $\nu_{\infty}$ and $\check{\nu}_{\infty}$ be the respective limit distributions on $\partial\Gamma$. The \textbf{Basic assumptions $(3.1)$} hold for $X_{n}$ and $\check{X}_{n}$, and the first item in the half-space method is fulfilled. 

Next, one of the main points of the method is to assign a strip $\mathfrak{s}(\mathfrak{u},\mathfrak{v})\subset\Gamma$ to almost every pair of ends $(\mathfrak{u},\mathfrak{v})\in\partial\Gamma\times\partial\Gamma$, and to prove that it satisfies the conditions from Proposition \ref{StripCriterion}. By the continuity of $\nu_{\infty}$  and $\check{\nu}_{\infty}$, the set $\{(\mathfrak{u},\mathfrak{v})\in\partial\Gamma\times\partial\Gamma:\mathfrak{u}=\mathfrak{v}\}$ has $(\nu_{\infty}\times\check{\nu}_{\infty})$-measure $0$, so that, in constructing the strip $\mathfrak{s}(\mathfrak{u},\mathfrak{v})$ we shall consider only the case $\mathfrak{u}\neq\mathfrak{v}$. For this, let $F$ be a $D$-cut, i.e. a finite subset of the Cayley graph of $\Gamma$, whose deletion disconnects $\Gamma$ into precisely two connected components. This cut  is used in defining the structure tree of the graph. For details, see Dicks and Dunwoody \cite{Dunwoody&Dicks}. Denote by $F^{0}$ the set of all end vertices (in the Cayley graph of $\Gamma$) of the edges of $F$. For every pair of ends $(\mathfrak{u},\mathfrak{v})\in\partial \Gamma\times\partial\Gamma$, let us define the strip 
\begin{equation*}
\mathfrak{s}(\mathfrak{u},\mathfrak{v})=\bigcup\{\gamma F^{0}:\ \gamma\in\Gamma:\ \widehat{\U}(\mathfrak{u},\gamma F)\neq \widehat{\U}(\mathfrak{v},\gamma F)\}.
\end{equation*}
The set $\U(\mathfrak{u},F)$ is the connected component which represents the end $\mathfrak{u}$ when we remove the finite set $F$ from $\Gamma$, and $\widehat{\U}(\mathfrak{u},F)$ is its completion (which contains $\mathfrak{u}$) in $\widehat{\Gamma}$. It is clear that $\mathfrak{s}(\mathfrak{u},\mathfrak{v})$ is a subset of $\Gamma$, and moreover $\gamma \mathfrak{s}(\mathfrak{u},\mathfrak{v})=\mathfrak{s}(\gamma\mathfrak{u},\gamma\mathfrak{v}) $, for every $\gamma\in\Gamma$. The strip $\mathfrak{s}(\mathfrak{u},\mathfrak{v})$ is the union of all $\gamma F^{0}$ such that the connected components $\widehat{\U}(\mathfrak{u},\gamma F)$ and $\widehat{\U}(\mathfrak{v},\gamma F)$, which contain the ends $\mathfrak{u}$ and $\mathfrak{v}$, respectively, when we remove the set $\gamma F$ from $\Gamma$, are not the same. In other words, $\mathfrak{s}(\mathfrak{u},\mathfrak{v})$ is the union of all $\gamma F^{0}$ such that the sides (one connected component of $\Gamma\setminus\gamma F$ and its complement in $\Gamma$) of $\gamma F$, seen as edges of the structure tree $\mathcal{T}$, lie on the geodesic between $\varphi\mathfrak{u}$ and $\varphi\mathfrak{v}$ ($\varphi$ is the structure map between $\partial\Gamma$ and its structure tree $\mathcal{T}$). This geodesic can be empty (when $\varphi(\mathfrak{u})=\varphi(\mathfrak{v})$), finite (when $\varphi(\mathfrak{u})$ and $\varphi(\mathfrak{v})$ are vertices in the structure tree $\mathcal{T}$), one way infinite or two way infinite. The latter holds when $\mathfrak{u},\mathfrak{v}$ are distinct thin ends (i.e. ends with finite diameter), and we have to check the subexponential growth of the strip only in this case. Using the properties of the structure tree of the Cayley graph of $\Gamma$, there is an integer $k>0$, such that the following holds: if $A_{0},A_{1},\ldots ,A_{k}$ are oriented edges in the structure tree $\mathcal{T}$ and connected components in $\Gamma$ such that $A_{0} \supset A_{1} \supset \cdots \supset A_{k}$ properly, then $d(A_{k},\Gamma\setminus A_{0})\geq 2$. Finiteness of $F^{0}$ implies that there is a constant $c>0$ such that 
\begin{equation*} \label{subexp}
|\mathfrak{s}(\mathfrak{u},\mathfrak{v})\cap B(e,n) |\leq cn,
\end{equation*}
for all $n$, and for all distinct thin ends $\mathfrak{u},\mathfrak{v}$. This proves the subexponential growth of $\mathfrak{s}(\mathfrak{u},\mathfrak{v})$.

Next, let us go to the partition of $\Gamma$ into half-spaces. By the definition of $\mathfrak{s}(\mathfrak{u},\mathfrak{v})$, every $x\in\mathfrak{s}(\mathfrak{u},\mathfrak{v})$ is contained in some cut $\gamma F$, for some $\gamma\in\Gamma$. Since a $D$-cut $F$ in a graph has the property that the sets $\gamma F$ do not intersect, for every $\gamma\in\Gamma$, it follows that every $x\in\mathfrak{s}(\mathfrak{u},\mathfrak{v})$ is contained in exactly one cut $\gamma F$. We partition $\Gamma$ in this way: for every $x\in\mathfrak{s}(\mathfrak{u},\mathfrak{v})$, we look at the $D$-cut $\gamma F$ containing $x$, and remove it from $\Gamma$. Then the set $(\Gamma\setminus\gamma F)$ contains precisely two connected components. This follows from the definition of a $D$-cut, and from the finiteness of the removed set $F$. Moreover, the connected components containing $\mathfrak{u}$ and $\mathfrak{v}$ are different, by the definiton of the strip. Let $\Gamma_{+}$ be the connected component of $(\Gamma\setminus\gamma F)$, which contains $\mathfrak{u}$ and $\Gamma_{-}$ be its complement in $\Gamma$, which contains $\mathfrak{v}$. One can see here that the partition of $\Gamma$ into the half-spaces $\Gamma_{+}$ and $\Gamma_{-}$ depends on the cut $\gamma F$ containing $x$, that is, depends on $x$. The sets $\Gamma_{+}$ and $\Gamma_{-}$ are $\Gamma$-equivariant. From the above, it follows that all the assumptions needed in the half-space method hold in the case of a group with infinitely many ends. Now, we apply Theorem \ref{PoissonTheorem}.

By Theorem \ref{ConvergenceTheorem} each of the random walks $Z_{n}$ and $\check{Z}_{n}$ starting at $id$ converges almost surely to an $\Omega$-valued random variable. If $\mu_{\infty}$  and $\check{\mu}_{\infty}$ are their respective limit distributions on $\Omega$, then the spaces $(\Omega,\mu_{\infty})$ and $(\Omega,\check{\mu}_{\infty})$ are $\mu$- and $\check{\mu}$- boundaries of the respective random walks. 

Take $b_{+}=(\phi_{+},\mathfrak{u})$, $b_{-}=(\phi_{-},\mathfrak{v})\in\Omega$, where $\phi_{+}$ and $\phi_{-}$ are the limit configurations of $Z_{n}$ and $\check{Z}_{n}$, respectively, and $\mathfrak{u},\mathfrak{v}\in\partial\Gamma$ are their only respective accumulation points. Define the configuration $\Phi(b_{+},b_{-},x)$ like in \eqref{StripConfiguration}, where $\Gamma\setminus(\Gamma_{+}\cup\Gamma_{-})$ is the empty set, and the strip $S(b_{+},b_{-})$ exactly like in \eqref{LamplighterStrip}. From Theorem \ref{PoissonTheorem}, $S(b_{+},b_{-})$ satisfies the conditions from Proposition \ref{StripCriterion}, and it follows that the space $(\Omega,\mu_{\infty})$ is the Poisson boundary of the lamplighter random walk $Z_{n}$ over $G=\Z_{2}\wr \Gamma$.
\end{proof}

\subsection{Hyperbolic groups}

Consider the group $\Gamma$ with the word metric $d(\cdot,\cdot)$ on it. The group $\Gamma$ is called (word) \textit{hyperbolic} if its Cayley graph corresponding to a finite generating set  $S$ is hyperbolic. A graph is called \textit{hyperbolic} (in the sense of Gromov) if there is a $\delta \geq 0$ such that every geodesic triangle in the graph is $\delta$-thin. We shall not lay out the basic features of hyperbolic graphs and their hyperbolic boundary and compactification, which we denote again by $\partial_{h}\Gamma$ and $\widehat{\Gamma}$. For details, the reader is invited to consult the texts by Gromov \cite{Gromov1987}, Ghys and de la Harpe \cite{GhysHarpe1990}, Coornaert, Delzant and Papadopoulus \cite{CoornaertDelzantPapadopoulos1990}, or, for a presentation in the context of random walks on graphs, Woess \cite[Section 22]{woess}.

If $\Gamma$ is a hyperbolic group, then one can understand how its hyperbolic compactification $\partial_{h}\Gamma$ is related to the end compactification $\partial\Gamma$: the former is finer, that is, the identity on $\Gamma$ extends to a continuos surjection from the hyperbolic to the end compactification  which maps $\partial_{h}\Gamma$ onto $\partial\Gamma$. For every two distinct points $\mathfrak{u},\mathfrak{v}$ on the hyperbolic boundary $\partial_{h}\Gamma$, there is a geodesic $\overline{\mathfrak{u}\mathfrak{v}}$ which may not be unique. The boundary of a finitely generated hyperbolic group $\Gamma$ is either infinite or has cardinality $2$. In the latter case, it is a group with two ends that is quasi-isometric with the two-way infinite path, and the Poisson boundary of any random walk with finite first moment is trivial. Thus, we assume that $\partial_{h}\Gamma$ is infinite, and consider lamplighter random walks $Z_{n}$ on $G=\Z_{2}\wr\Gamma$, when $\Gamma$ is a hyperbolic group. The natural geometric boundary of $\Gamma$ is the hyperbolic boundary $\partial_{h}\Gamma$. For the Poisson boundary of the lamplighter random walk  $Z_{n}$ on $G=\Z_{2}\wr\Gamma$ we have to distinguish two cases: \textbf{(a)} infinite hyperbolic boundary and infinitely many ends and \textbf{(b)} infinite hyperbolic boundary and only one end.

\paragraph{(a) Infinite hyperbolic boundary and infinitely many ends.} From the fact that the identity on $\Gamma$ extends to a continuous surjection from the hyperbolic to the end compactification, which maps $\partial_{h}\Gamma$ onto $\partial\Gamma$, it follows that, this is exactly the case treated in Section \ref{GraphsWithInfEnds}, i.e., of a group with infinitely many ends, where the ends are the connected components of the hyperbolic boundary. The Poisson boundary of the lamplighter random walk $Z_{n}$ is given by Theorem \ref{PoissonInfEnds}, with the hyperbolic boundary $\partial_{h}\Gamma$ instead of the space of ends $\partial\Gamma$.

\paragraph{(b) Infinite hyperbolic boundary and only one end.} What we want is to determine the Poisson boundary of lamplighter random walks $Z_{n}$ on $G=\Z_{2}\wr\Gamma$, when $\Gamma$ is a finitely generated hyperbolic group with infinite boundary and only one end. In order to use the half-space method defined before, we shall need some additional definitions.

Consider the hyperbolic boundary $\partial_{h}\Gamma$ as being described by equivalence of geodesic rays. For $y\in\Gamma$ and $\mathfrak{u}\in\partial_{h}\Gamma$ let $\pi=[y=y_{0},y_{1},\ldots,\mathfrak{u}]$ be a geodesic ray joining $y$ with $\mathfrak{u}$. For every $x\in\Gamma$ let $\beta_{\mathfrak{u}}(x,\pi)=\limsup_{i\to\infty}(d(x,y_{i})-i)$, and define the \textit{Busemann function} of the point $\mathfrak{u}\in\partial_{h}\Gamma$, $\beta_{\mathfrak{u}}:\Gamma\times\Gamma\rightarrow \mathbb{R}$ as follows:
\begin{equation*}
\beta_{\mathfrak{u}}(x,y)=\sup\{\beta_{\mathfrak{u}}(x,\pi^{'}):\pi^{'} \mbox{ is a geodesic ray from }y\mbox{ to } \mathfrak{u}\}.
\end{equation*}
The \textit{horosphere} with the centre $\mathfrak{u}$ passing through $x\in\Gamma$, denoted $H_{x}(\mathfrak{u})$ is the set
\begin{equation*}
H_{x}(\mathfrak{u})=\{ y\in\Gamma:\ \beta_{\mathfrak{u}}(x,y)=0\}.
\end{equation*}
The following result is another special case of Theorem \ref{PoissonTheorem}.
\begin{theo}\label{PoissonHyperbolicGraphs}
Let $\Gamma$ be a finitely generated hyperbolic group with infinite hyperbolic boundary and only one end, and $Z_{n}=(\eta_{n},X_{n})$ be a random walk with law $\mu$ on $G=\Z_{2}\wr\Gamma$, such that $supp(\mu)$ generates $G$. Suppose that $\mu$ has finite first moment and $\Omega$ is defined as in \eqref{omega}, with the hyperbolic boundary $\partial_{h}\Gamma$ instead of $\partial\Gamma$. If $\mu_{\infty}$ is the limit distribution on $\Omega$ of the random walk $Z_{n}$ starting at $id=\2$, then $(\Omega,\mu_{\infty})$ is the Poisson boundary of the random walk.
\end{theo}
\begin{proof} The proof is as in the preceding example. First of all, we show that the conditions required in the half-space method are satisfied for $\Gamma$ and for random walks $X_{n}$ on $\Gamma$. One can check that the hyperbolic boundary $\partial_{h}\Gamma$ satisfies the convergence property (\ref{Convergence}). When $\Gamma$ is a hyperbolic group, and the law $\nu$ of the random walk $X_{n}$ on $\Gamma$ has finite first moment, Woess \cite{WoessFixedSets1993} (see also Woess \cite{woess}) proved that $X_{n}$ converges almost surely in the hyperbolic topology to a random element from $\partial_{h}\Gamma$, and the limit distribuition is a continuous measure. The same is true for the random walk $\check{X}_{n}$ with law $\check{\nu}$ on $\Gamma$. Let $\nu_{\infty}$ and $\check{\nu}_{\infty}$ be the respective limit distributions on $\partial_{h}\Gamma$. The \textbf{Basic assumptions $(3.1)$} hold for $X_{n}$ and $\check{X}_{n}$, and the first item in the half-space method is fulfilled. 

In order to prove that the second item in the half-space method holds, we assign to almost every pair of boundary points $(\mathfrak{u},\mathfrak{v})\in\partial_{h}\Gamma\times\partial_{h}\Gamma$ a strip $\mathfrak{s}(\mathfrak{u},\mathfrak{v})\subset\Gamma$ and we prove that it satisfies the conditions from Proposition \ref{StripCriterion}. By the continuity of $\nu_{\infty}$  and $\check{\nu}_{\infty}$ on $\partial_{h}\Gamma$, the set $\{(\mathfrak{u},\mathfrak{v})\in\partial_{h}\Gamma\times\partial_{h}\Gamma:\mathfrak{u}=\mathfrak{v}\}$ has $(\nu_{\infty}\times\check{\nu}_{\infty})$-measure $0$, so that, in constructing $\mathfrak{s}(\mathfrak{u},\mathfrak{v})$ we consider only the case $\mathfrak{u}\neq\mathfrak{v}$. Let 
\begin{equation*}
\mathfrak{s}(\mathfrak{u},\mathfrak{v})=\bigcup\{x\in\Gamma:x\mbox{ lies on a two way infinite geodesic between }\mathfrak{u}\mbox{ and }\mathfrak{v}\}.
\end{equation*} 
The strip $\mathfrak{s}(\mathfrak{u},\mathfrak{v})$ is the union of all points $x$ from all geodesics in $\Gamma$ joining $\mathfrak{u}$ and $\mathfrak{v}$. This is a subset of $\Gamma$, and $\gamma \mathfrak{s}(\mathfrak{u},\mathfrak{v})=\mathfrak{s}(\gamma\mathfrak{u},\gamma\mathfrak{v}) $, for every $\gamma\in\Gamma$. Since in a hyperbolic space any two geodesics with the same endpoints are within uniformly bounded distance one from another (see \cite{GhysHarpe1990} for details), and the geodesics have linear growth, it follows that 
there exists a constant $c>0$ such that 
\begin{equation*}
|\mathfrak{s}(\mathfrak{u},\mathfrak{v})\cap B(e,n)|\leq cn,
\end{equation*}
for all $n$ and distinct $\mathfrak{u},\mathfrak{v}\in\partial_{h}\Gamma$, and this proves the subexponential growth of $\mathfrak{s}(\mathfrak{u},\mathfrak{v})$.

Finally, let us partition $\Gamma$ into half-spaces. Actually, this is one of the examples where the partition is made in two half-spaces and another ``not-interesting'' set on which the configuration will be $0$. For every $x\in\mathfrak{s}(\mathfrak{u},\mathfrak{v})$, let $H_{x}(\mathfrak{u})$ (respectively, $H_{x}(\mathfrak{v})$) be the horosphere with center $\mathfrak{u}$ (respectively, $\mathfrak{v}$) and passing through $x$. Remark that the two horospheres may have non empty intersection. Consider the partition of $\Gamma$ into the subsets $\Gamma_{+}$, $\Gamma_{-}$, and $\Gamma\setminus(\Gamma_{+}\cup\Gamma_{-})$, where $\Gamma_{+}=H_{x}(\mathfrak{u})$ (that is, it contains a neighbourhood of $\mathfrak{u}$) and $\Gamma_{-}=H_{x}(\mathfrak{v})\setminus H_{x}(\mathfrak{u})$. This partition is $\Gamma$-equivariant. From the above, it follows that all the assumptions needed in the half-space method hold in the case of a finitely generated hyperbolic group $\Gamma$. Now, we apply Theorem \ref{PoissonTheorem}.

By Theorem \ref{ConvergenceTheorem} each of the random walks $Z_{n}$ and $\check{Z}_{n}$ starting at $id$ converges almost surely to an $\Omega$-valued random variable, where $\Omega$ is defined as in \eqref{omega}, with the hyperbolic boundary $\partial_{h}\Gamma$ instead of $\partial\Gamma$. If $\mu_{\infty}$  and $\check{\mu}_{\infty}$ are their respective limit distributions on $\Omega$, then the spaces $(\Omega,\mu_{\infty})$ and $(\Omega,\check{\mu}_{\infty})$ are $\mu$- and $\check{\mu}$- boundaries of the respective random walks. 

Take $b_{+}=(\phi_{+},\mathfrak{u})$, $b_{-}=(\phi_{-},\mathfrak{v})\in\Omega$, where $\phi_{+}$ and $\phi_{-}$ are the limit configurations of $Z_{n}$ and $\check{Z}_{n}$, respectively, and $\mathfrak{u},\mathfrak{v}\in\partial\Gamma$ are their only respective accumulation points. Define the configuration $\Phi(b_{+},b_{-},x)$ like in \eqref{StripConfiguration}, and the strip $S(b_{+},b_{-})$ exactly like in \eqref{LamplighterStrip}. From Theorem \ref{PoissonTheorem}, $S(b_{+},b_{-})$ satisfies the conditions from Proposition \ref{StripCriterion}, and it follows that the space $(\Omega,\mu_{\infty})$ (with the hyperbolic boundary $\partial_{h}\Gamma$ instead of $\partial\Gamma$ in the definition \eqref{omega} of $\Omega$) is the Poisson boundary of the lamplighter random walk $Z_{n}$ over $G=\Z_{2}\wr \Gamma$.
\end{proof}

\subsection{Euclidean lattices}

Let now $\Gamma= \Z^{d}$, $d\geq 3$, be the $d$-dimensional lattice, with the Euclidean metric $|\cdot|$ on it. For $\Z^{d}$, there are also natural boundaries and compactifications. A nice example of compactification is obtained by embedding $\Z^{d}$ into the $d$-dimensional unit disc via the map $x\mapsto x/(1+|x|)$, and taking the closure. In this compactification, the boundary $\partial\Z^{d}$ is the unit sphere $S_{d-1}$ in $\mathbb{R}^d$, and a sequence $x_{n}$ in $\Z^{d}$ converges to $\mathfrak{u}\in S_{d-1}$ if and only if $|x_{n}|\to\infty$ and $x_{n}/|x_{n}|\to\mathfrak{u}$, as $n\to\infty$. 

Let us check that the property \eqref{Convergence} holds for the boundary $S_{d-1}$. For this, let $x_{n}$ a sequence converging to $\mathfrak{u}\in S_{d-1}$, and $y_{n}$ another sequence in $\Z_{d}$, such that $x_{n}/|x_{n}|\to\mathfrak{u}$ and $|x_{n}-y_{n}|/|x_{n}|\to 0$ as $n\to\infty$. Since 
\begin{equation*}
\Big{|}\dfrac{y_{n}}{|x_{n}|}-\dfrac{x_{n}}{|x_{n}|}\Big{|}\leq\dfrac{|x_{n}-y_{n}|}{|x_{n}|}\to 0,
\end{equation*}
it follows that $y_{n}/|x_{n}|\to\mathfrak{u}$. Now, $y_{n}/|x_{n}|=(y_{n}/|y_{n}|)\cdot(|y_{n}|/|x_{n}|)$, and the sequence $|y_{n}|/|x_{n}|$ of real numbers converges to $1$, since we can bound it from above and from below by two sequences both converging to 1. Therefore $y_{n}/|y_{n}|\to\mathfrak{u}$, and the property \eqref{Convergence} holds. Next, if the law $\nu$ of the random walk $X_{n}$ on $\mathbb{Z}^{d}$ has non-zero first moment (drift)
\begin{equation*}
m=\sum_{x}x\nu(x)\in\mathbb{R}^{d}, 
\end{equation*}
then the law of large numbers implies that $X_{n}$ converges to the boundary $S_{d-1}$ in this compactification with deterministic limit $m/|m|$. In particular, the limit distribution $\nu_{\infty}$ is the Dirac mass at this point. Next, we state the result on the Poisson boundary of lamplighter random walks over $\Z^{d}$ in the case of non-zero drift, using the half-space method. Remark that, in this case the description of the Poisson boundary war earlier obtained by Kaimanovich \cite{KaimanovichPreprint}.
\begin{theo}\label{PoissonEuclideanLattices}
Let $\Gamma=\mathbb{Z}^{d}$, $d\geq 3$ be a Euclidean lattice and $Z_{n}=(\eta_{n},X_{n})$ be a random walk with law $\mu$ on $G=\Z_{2}\wr \mathbb{Z}^{d}$, such that $supp(\mu)$ generates $G$, and the projected random walk $X_{n}$ on $\mathbb{Z}^{d}$ has non-zero drift. Suppose that $\mu$ has finite first moment and $\Omega$ is defined as in \eqref{omega}, with the unit sphere $S_{d-1}$ instead of $\partial\Gamma$. If $\mu_{\infty}$ is the limit distribution on $\Omega$ of the random walk $Z_{n}$ starting at $id=\2$, then $(\Omega,\mu_{\infty})$ is the Poisson boundary of the random walk.
\end{theo}
\begin{proof} Let us show that the conditions required in the half-space method are satisfied for $\Gamma=\Z^{d}$ and for random walks $X_{n}$ on $\Gamma$. The random walk $X_{n}$ (respectively $\check{X}_{n}$) converges to the boundary $S_{d-1}$ with deterministic limit $\mathfrak{u}=m/|m|$ (respectively, $\mathfrak{v}=-m/|m|$), in the case of non-zero mean $m$, and the convergence property \eqref{Convergence} holds. The \textbf{Basic assumptions $(3.1)$} hold for $X_{n}$ and $\check{X}_{n}$, and the first item in the half-space method is satisfied. The limit distributions $\nu_{\infty}$ and $\check{\nu}_{\infty}$ are the Dirac-masses at this limit points.

Now, to define a strip in $\Z^{d}$ is an easy task, because of the growth of $\Z^{d}$. For the two limit  points $\mathfrak{u}$ and $\mathfrak{v}$ of $X_{n}$ and $\check{X}_{n}$, respectively, define the strip $\mathfrak{s}(\mathfrak{u},\mathfrak{v})=\Z^{d}$. This strip does not depend on the limit points, it is $\Z^{d}$-equivariant, and it has polynomial growth of order $d$, that is, also subexponential growth. Next, let us partition $\Z^{d}$ into half-spaces. Denote by $\overline{\mathfrak{u}\mathfrak{v}}$ the geodesic of $S_{d-1}$ joining the two deterministic boundary points $\mathfrak{u},\mathfrak{v}\in S_{d-1}$. In this case, this is exactly the diameter in the ball, since the points $\mathfrak{u}$ and $\mathfrak{v}$ are antipodal points, i.e. they are opposite through the centre. For every $x\in\mathfrak{s}(\mathfrak{u},\mathfrak{v})=\Z^{d}$, consider the hyperplane which passes through $x$ and is orthogonal to $\overline{\mathfrak{u}\mathfrak{v}}$. This hyperplane cuts $\Z^{d}$ into two disjoint spaces $\Gamma_{+}$ and $\Gamma_{-}$, containing $\mathfrak{u}$ and $\mathfrak{v}$, respectively. Hence, $\Gamma=\Z^{d}$ is partitioned into half-spaces $\Gamma_{+}$ and $\Gamma_{-}$, which are $\Z^{d}$-equivariant. From the above, it follows that all the assumptions needed in the half-space method hold in the case of a Euclidean lattice $\Z^{d}$. Now, we apply Theorem \ref{PoissonTheorem}.

By Theorem \ref{ConvergenceTheorem} each of the random walks $Z_{n}$ and $\check{Z}_{n}$ starting at $id$ converges almost surely to an $\Omega$-valued random variable, where $\Omega$ is defined as in \eqref{omega}, with $S_{d-1}$ instead of $\partial\Gamma$. Nevertheless, the only ``active'' points of non-zero $\nu_{\infty}$- and $\check{\nu}_{\infty}$-measure on $S_{d-1}$ are  $\mathfrak{u}=m/|m|$ and $\mathfrak{v}=-m/|m|$, respectively. More precisely, $\Omega$ can be written as
\begin{equation}\label{EuclideanOmega}
\Omega =\Big{(}\mathcal{C}_{\mathfrak{u}}\times\{\mathfrak{u}\}\Big{)}\cup\Big{(}\mathcal{C}_{\mathfrak{v}}\times\{\mathfrak{v}\}\Big{)}, 
\end{equation}
where $\mathcal{C}_{\mathfrak{u}}$ (respectively, $\mathcal{C}_{\mathfrak{v}}$) is the set of all configurations accumulating only at $\mathfrak{u}$ (respectively, $\mathfrak{v}$).

If $\mu_{\infty}$  and $\check{\mu}_{\infty}$ are the limit distributions of $Z_{n}$ and $\check{Z}_{n}$ on $\Omega$, then the spaces $(\Omega,\mu_{\infty})$ and $(\Omega,\check{\mu}_{\infty})$ are $\mu$- and $\check{\mu}$- boundaries of the respective random walks. Take $b_{+}=(\phi_{+},\mathfrak{u})$, $b_{-}=(\phi_{-},\mathfrak{v})\in\Omega$, where $\phi_{+}$ and $\phi_{-}$ are the limit configurations of $Z_{n}$ and $\check{Z}_{n}$, respectively, and $\mathfrak{u},\mathfrak{v}\in\partial\Gamma$ are their only respective accumulation points. Define the configuration $\Phi(b_{+},b_{-},x)$ like in \eqref{StripConfiguration}, and the strip $S(b_{+},b_{-})$ exactly like in \eqref{LamplighterStrip}. From Theorem \ref{PoissonTheorem}, $S(b_{+},b_{-})$ satisfies the conditions from Proposition \ref{StripCriterion}, and it follows that the space $(\Omega,\mu_{\infty})$, with $\Omega$ as in \eqref{EuclideanOmega} is the Poisson boundary of the lamplighter random walk $Z_{n}$ over $G=\Z_{2}\wr \Gamma$.
\end{proof}
\paragraph*{Final remarks.}
One can also apply the method in order to find the Poisson boundary of lamplighter random walks over polycyclic groups, nilpotent groups, discrete groups of semi-simple Lie groups. Another application of the method is the determination of
the Poisson boundary of random walks over ``iterated'' lamplighter groups. That is, we consider our base group as being $\Z_{2}\wr\Gamma$ and we construct a new lamplighter group $\Z_{r}\wr(\Z_{r}\wr\Gamma)$ over $\Z_{r}\wr\Gamma$. The interesting fact here is that, the geometry of the group $\Z_{r}\wr\Gamma$ is completely different from that one of $\Gamma$. For instance, when $\Gamma$ is a group with infinitely many ends, $\Z_{r}\wr\Gamma$ has only one end. Our method still works, that is, we start with a strip in the lamplighter graph  $\Z_{r}\wr X$ and we lift it to a ''bigger`` one in $\Z_{r}\wr (\Z_{r}\wr X)$, following the steps of our method.

\paragraph*{Acknowledgements}
I am grateful to Wolfgang Woess for numerous numerous fruitful disscusions and for his help during the writing of this manuscript, and also to Vadim Kaimanovich for several hints and useful remarks regarding the content and exposition. I would also like to thank to the referee for careful reading, suggestions and corrections that helped to improve the paper.

\begin{small}\end{small}

\end{document}